\documentclass[a4paper,12pt,thmsa]{amsart}
\usepackage[a4paper,marginratio={1:1},scale={0.72,0.74},footskip=7mm,headsep=10mm]{geometry}

\usepackage{amsfonts}
\usepackage{amssymb,amsmath,latexsym}
\usepackage[dvips]{graphics}
\usepackage{graphicx,subfigure}
\usepackage[dvips]{color}
\usepackage[T1]{fontenc}
\usepackage[active]{srcltx}
\usepackage{amsmath}
\usepackage{amsfonts}
\usepackage{amssymb}
\usepackage{psfrag}
\usepackage{color}
\usepackage{url}
\usepackage{amsthm}
\usepackage{array}
\usepackage{pst-tree}
\usepackage{lscape}
\usepackage[T1]{fontenc}
\usepackage{pstricks,pstricks-add}
\usepackage{mathrsfs}  

\newtheorem{theorem}{Theorem}[section]
\newtheorem{proposition}{Proposition}[section]
\newtheorem{remark}{Remark}[section]
\newtheorem{definition}{Definition}[section]
\newtheorem{corollary}{Corollary}[section]

\newtheorem{lemma}{Lemma}[section]

\def\e{\mathbb{E}}
\def\p{\mathbb{P}}
\newcommand{\ind}{\mbox{\rm 1\hspace{-0.04in}I}}

\title[On distributions determined by their upward Wiener-Hopf factor]
{On distributions determined by their upward, space-time Wiener-Hopf factor}

\author{Lo\"ic Chaumont}

\address{Lo\"ic Chaumont -- LAREMA -- UMR CNRS 6093, Universit\'e d'Angers, 2 bd Lavoisier, 49045 Angers cedex~01}

\email{loic.chaumont@univ-angers.fr}

\author{Ron Doney}

\address{R.A. Doney Department of Mathematics, University of Manchester, Manchester,
M 13 9 PL.}

\email{ron.doney@manchester.ac.uk}



\thanks{}

\date{\today}

\begin{document}

\begin{abstract} According to the Wiener-Hopf factorization, the characteristic function $\varphi$ of any probability distribution 
$\mu$ on $\mathbb{R}$ can be decomposed in a unique way as 
\[1-s\varphi(t)=[1-\chi_-(s,it)][1-\chi_+(s,it)]\,,\;\;\;|s|\le1,\,t\in\mathbb{R}\,,\]
where  $\chi_-(e^{iu},it)$ and $\chi_+(e^{iu},it)$ are the characteristic functions of possibly defective distributions 
in $\mathbb{Z}_+\times(-\infty,0)$ and $\mathbb{Z}_+\times[0,\infty)$, respectively. 

We prove that $\mu$ can be characterized by the sole data of the upward factor 
$\chi_+(s,it)$, $s\in[0,1)$, $t\in\mathbb{R}$ in many cases including the cases where: 

1) $\mu$ has some exponential moments; 

2) the function $t\mapsto\mu(t,\infty)$ is completely monotone on $(0,\infty)$; 

3) the density of $\mu$ on $[0,\infty)$ admits an analytic continuation on $\mathbb{R}$. 

We conjecture that any probability distribution is actually characterized by its upward factor. 
This conjecture is equivalent to the following: {\it Any probability measure $\mu$ on $\mathbb{R}$ whose 
support is not included in $(-\infty,0)$ is determined by its convolution powers $\mu^{*n}$, 
$n\ge1$ restricted to $[0,\infty)$}. We show that in many instances, the sole knowledge of $\mu$ and $\mu^{*2}$ 
restricted to $[0,\infty)$ is actually sufficient to determine $\mu$. Then we investigate the analogous problem in
the framework of infinitely divisible distributions. 
\end{abstract}

\maketitle

\section{Introduction}\label{int}

Let $\mu$ be any probability measure on $\mathbb{R}$. Denote by $(S_n)$ a random walk with step distribution 
$\mu$, such that $S_0=0$, a.s. Define the first ladder times associated to $(S_n)$ by  
\[\tau_-=\inf\{n\ge1:S_n<0\}\,,\;\;\;\tau_+=\inf\{n\ge1:S_n\ge0\}\,.\]
Then the Wiener-Hopf factorization of the characteristic function $\varphi(t)=\int e^{itx}\mu(dx)$ of $\mu$ can be written
as,
\begin{equation}\label{6283}
1-s\varphi(t)=[1-\chi_-(s,it)][1-\chi_+(s,it)]\,,\;\;\;|s|\le1,\,t\in\mathbb{R}\,,
\end{equation}
where $\chi_-$ and $\chi_+$ are the downward and upward space-time Wiener-Hopf factors,
\[\chi_-(s,it)=\e(s^{\tau_-}e^{itS_{\tau_-}}1_{\{\tau_-<\infty\}})\;\;\;\mbox{and}\;\;\;\chi_+(s,it)=
\e(s^{\tau_+}e^{itS_{\tau_+}}1_{\{\tau_+<\infty\}})\,.\]
To paraphrase W. Feller \cite{fe}, XVIII.3, {\it the remarkable feature of the factorization $(\ref{6283})$ is that it represents 
an arbitrary characteristic function $\varphi$ in terms of two (possibly defective) distributions, one being concentrated on the 
half line $(-\infty,0)$ and the other one on the half line  $[0,\infty)$}. However, this feature only exploits identity (\ref{6283}) for 
fixed $s\neq0$ and reflects the fact that $\mu$ is determined by the knowledge of the distributions of both $S_{\tau_-}$ and 
$S_{\tau_+}$. But one may wonder about the extra information brought by the joint distributions $(\tau_-,S_{\tau_-})$ and 
$(\tau_+,S_{\tau_+})$. In particular, is it true in general that $\mu$ is determined by only one of these joint distributions? 
or equivalently, is it true that $\varphi$ is determined by only one of its space-time Wiener-Hopf factors?\\

The aim of this paper is an attempt to answer the latter question. We will actually show that $\mu$ is determined by $\chi_+(s,it)$ 
in some quite large classes of distributions including the case where $\mu$ has some positive exponential moments, or when 
$t\mapsto\mu(t,\infty)$ is completely monotone on $(a,\infty)$, for some $a\ge0$, or satisfies some property which is slightly stronger 
than analyticity. Obviously all these assumptions can be verified from the sole data of $\chi_+(s,it)$. These different cases cover a 
sufficiently large range of distributions for us to allow ourselves to raise the following conjecture. Let $\mathcal{M}_1$ be 
the set of probability measures on $\mathbb{R}$.\\

\noindent {\bf Conjecture C.} {\it  Any distribution $\mu\in\mathcal{M}_1$ whose support is not included in $(-\infty,0)$ is 
determined by its upward space-time Wiener-Hopf factor $\chi_+(s,it)$, $|s|<1$, $t\in\mathbb{R}$.}\\

\noindent A crucial step in the proof of (\ref{6283}) is the following development of the factor $\chi_+(s,t)$, for $|s|<1$ and
$t\in\mathbb{R}$,
\begin{equation}\label{1476}
\log\frac1{1-\chi_+(s,it)}=\sum_{n=1}^\infty\frac{s^n}n\int_{[0,\infty)}e^{itx}\mu^{*n}(dx)\,,
\end{equation}
see \cite{fe}, XVIII.3, where $\mu^{*n}$ is the nth fold convolution product of $\mu$ by itself. We will actually refer to 
$\mu^{*n}$, $n\ge0$ as the convolution powers of $\mu$. This identity shows that the 
data of $\chi_+$ is equivalent to the knowledge of the measures $\mu^{*n}$, $n\ge1$ on $[0,\infty)$ and leads to the 
following equivalent conjecture.\\

\noindent {\bf Conjecture C'.} {\it Any distribution  $\mu\in\mathcal{M}_1$  whose support is not included in $(-\infty,0)$ is 
determined by its convolution powers $\mu^{*n}$, $n\ge1$ restricted to $[0,\infty)$.}\\

Each of the next sections corresponds to a class of probability distributions for which Conjecture C holds. For the first one in 
Section \ref{FourierLaplace}, we prove that distributions having some particular exponential moments 
satisfy conjecture C. Then in Section \ref{5673} we consider three other classes for which a much 
stronger result than Conjectures C and C' is true. We will see that there are actually many distributions which are determined by 
the sole data of $\mu$ and $\mu^{*2}$ on $[0,\infty)$. This is the case when the function $t\mapsto\mu(t,\infty)$ is smooth 
enough. In Subsection \ref{monotone} we will consider the case where the function $t\mapsto\mu(t,\infty)$ is completely 
monotone on $(a,\infty)$, for some $a\ge0$ and in Subsection \ref{analytic} we will make a slightly stronger assumption than 
analyticity on this function. We will also present the discrete counterpart of the completely monotone case in Subsection \ref{6288}. 
Finally in Section \ref{2482} we will consider Conjecture C in the restricted class of infinitely divisible distributions and show that if 
the upper tail of the L\'evy measure is completely monotone, then $\mu$ is determined by its upper Wiener-Hopf factor. Then we 
end this paper in Section \ref{2450} with some important remarks on the possibility of extending the classes of distributions 
studied. 

Throughout this paper, we will denote by  $\mathscr{C}$ the set of distributions satisfying Conjecture C. Let us give a proper 
definition of this set.

\begin{definition}
Let $\mathscr{C}$ be the set of distributions $\mu\in\mathcal{M}_1$ which are
determined by their upward Wiener-Hopf factor $\chi_+(s,it)$, for $|s|\le1$ and $t\in\mathbb{R}$ or equivalently by the data of 
their convolution powers $\mu^{*n}$, $n\ge1$ restricted to $[0,\infty)$. More formally,
\[\mathscr{C}=\{\mu\in\mathcal{M}_1:\mbox{if $\mu_1\in\mathcal{M}_1$
satifies $\mu^{*n}=\mu_1^{*n}$, on $[0,\infty)$, for all $n\ge0$, then $\mu=\mu_1$}\}\,.\]
\end{definition}

The problem we investigate here originates from a result in V.~Vigon's PhD thesis \cite{vi}, see Section 4.5 therein, where a question 
equivalent to Conjecture C is raised in the setting of  L\'evy processes. Our question is actually more general since it concerns distributions 
which are not necessarily infinitely divisible. In particular, a positive answer to Conjecture C would imply that the law of any L\'evy process 
$(X_t,t\ge0)$  is determined by one of its space-time Wiener-Hopf factors or equivalently by the marginals of the process $(X_t^+,t\ge0)$, 
see Section \ref{2482}.

\section{When $\mu$ admits exponential moments.}\label{FourierLaplace} 

\subsection{Recovering the characteristic function and the moment generating function.}
In this paper, we will always assume that the support of the measure $\mu$ is not included in $(-\infty,0)$. 
Let us observe that from the data of the upward Wiener-Hopf factor $\chi_+(s,it)$, for $|s|\le1$ and $t\in\mathbb{R}$ or 
equivalently from the data of the measures $\mu^{*n}$, $n\ge1$ restricted to $[0,\infty)$, we know the sequences, $\p(S_n<0)$ 
and $\p(S_n\ge0)$, $n\ge0$, as well as the distributions of both $\tau_-$ and $\tau_+$. In particular we know whether $(S_n)$ 
oscillates, drifts to $-\infty$, or drifts to $\infty$.  The next result shows that provided $n\mapsto\p(S_n<0)$ tends to 0 sufficiently 
fast along some subsequence, it is possible to recover the characteristic function $\varphi$ of $\mu$ on some interval containing 
0, from the measures $\mu^{*n}$ restricted to $[0,\infty)$.

\begin{lemma}\label{2361} 
Assume that there is $\alpha>0$ such that, at least for a subsequence, 
\begin{equation}\label{0236}
\p(S_{n}<0)\le e^{-\alpha n}\,.
\end{equation}
Then for all $t$ such that $|\varphi(t)|>e^{-\alpha}$, along this subsequence,
\[\lim_{n\rightarrow+\infty}\e(e^{itS_{n}}\ind_{\{S_n\ge0\}})^{1/{n}}=\varphi(t)\,.\]
In particular, if $(\ref{0236})$ holds then $\varphi$ can be determined on some neighborhood of $0$. 
\end{lemma}
\begin{proof} Recall that $\varphi(t)$ tends to 1 as $t$ tends to 0. Then let $t$ be sufficiently small so that $\varphi(t)\neq0$ and 
let us write,
\begin{eqnarray}
\e(e^{it S_n}\ind_{\{S_n\ge0\}})^{1/n}&=&\left[\varphi(t)^n-\e(e^{it S_n}\ind_{\{S_n<0\}})\right]^{1/n}\nonumber\\
&=&\varphi(t)[1-\varphi(t)^{-n}\e(e^{it S_n}\ind_{\{S_n<0\}})]^{1/n}\,.\label{5454}
\end{eqnarray}
From the assumption, whenever is $t$ such that $|\varphi(t)|>e^{-\alpha}$ and for all $n$ such that $\p(S_n<0)<e^{-\alpha n}$,
\begin{eqnarray*}
 \left|\varphi(t)^{-n}\e(e^{it S_n}\ind_{\{S_n<0\}})\right|&\le&| \varphi(t)^{-n}\p(S_n<0)|\,,\\
 &\le&\left(|\varphi(t)|e^{\alpha}\right)^{-n}\,.
 \end{eqnarray*}
Therefore, the left hand side of the above inequality tends to 0 along a subsequence and this yields the result thanks to 
equation (\ref{5454}). 
\end{proof}
\noindent Since, for all $n\ge0$,
\[\p(S_1<0,S_2-S_1<0,\dots,S_n-S_{n-1}<0)=\p(S_1<0)^n\le\p(S_n<0)\,,\] 
$(\ref{0236})$ cannot hold for all $\alpha>0$, unless $\p(S_1\ge0)=1$. Note also that if $(\ref{0236})$ holds then the random 
walk $(S_n)$ cannot drift to $-\infty$. Moreover, if it holds for all suffciently large $n$, then $(S_n)$ necessarily drifts to $\infty$
thanks to Spitzer's criterion which asserts that this happens if and only if $\sum n^{-1}\p(S_n<0)<\infty$.\\

\begin{lemma}\label{3478} 
Let $\mu_1,\mu_2\in\mathcal{M}_1$. Denote by $\varphi_1$ and $\varphi_2$ their characteristic functions
and by $\phi_1$ and $\phi_2$ their moment generating functions, that is $\varphi_j(u):=\int_{\mathbb{R}}e^{iu x}\,\mu_j(dx)$
and $\phi_j(v):=\int_{\mathbb{R}}e^{v x}\,\mu_j(dx)$, $u,v\in\mathbb{R}$, $j=1,2$. Assume that there 
exists $\lambda>0$ such that $\phi_1(\lambda)<\infty$ and $\phi_2(\lambda)<\infty$ and that there is an open interval $I$ 
such that $\varphi_1(u)=\varphi_2(u)$, for all $u\in I$. Then $\mu_1=\mu_2$. 
\end{lemma}
\begin{proof} 
Let $D=\{z=u+iv\in\mathbb{C}:u\in\mathbb{R},\,-\lambda<v<0\}$. From the assumptions, the function $f:=\varphi_1-\varphi_2$ admits 
an analytic continuation in the open domain $D$. Then let $O_+=\{z=u+iv\in\mathbb{C}:u\in I,\,0<v<\lambda\}$ and 
$O_-=\{z=u+iv\in\mathbb{C}:u\in I,\,-\lambda<v<0\}$. From Schwarz reflection principle, $f$ admits an analytic continuation in the open 
domain $O_+\cup I \cup O_-$. From the principle of isolated zeroes, $f$ vanishes in $O_+\cup I \cup O_-$ and from the same principle, 
it vanishes in $D$. By continuity, $f(u)=0$ for all $u\in\mathbb{R}$ and the result follows from injectivity of the Fourier 
transform.\\
\end{proof}

\noindent  We will say that a distribution $\mu\in\mathcal{M}_1$ admits an exponential moment if there is 
$\lambda\in\mathbb{R}\setminus\{0\}$ such that $\phi(\lambda):=\int_{\mathbb{R}}e^{\lambda x}\,\mu(dx)<\infty$.\\

\begin{lemma}\label{3024} 
The characteristic function of a distribution having an exponential moment cannot vanish identically in an interval. 
\end{lemma}
\begin{proof} 
Let $\varphi$ be the characteristic function of $\mu\in\mathcal{M}_1$. If $\varphi$ vanishes in an interval and if  there is 
$\lambda\in\mathbb{R}\setminus\{0\}$ such that $\phi(\lambda)<\infty$, then we conclude
from the same arguments as in the proof of Lemma \ref{3478} (replacing $f$ by $\varphi$) that $\varphi(u)=0$ for all 
$u\in\mathbb{R}$, which is absurd since $\varphi$ is a characteristic function. 
\end{proof}

\noindent Lemma \ref{3024} was noticed in \cite{sm} for nonnegative random variables.\\

\begin{theorem}\label{4572} 
Assume that we know the measures $\mu^{*n}$ restricted to $[0,\infty)$. Then we can determine if $(\ref{0236})$ holds. 
Moreover if $(\ref{0236})$ holds, then $\mu$ belongs to the class $\mathscr{C}$. 
\end{theorem}
\begin{proof} 
The first assertion is trivial since $1-\mu^{*n}[0,\infty)=\p(S_{n}<0)$. Now if $(\ref{0236})$ holds, then from Lemma \ref {2361}, 
the characteristic function $\varphi$ of $\mu$ can be determined on some neighborhood of 0. Since the function 
$t\mapsto \chi_+(1,it)$ is known, from (\ref{6283}), it means that the function $t\mapsto \chi_-(1,it)$ can be determined on the 
same neighborhood. But $t\mapsto\chi_-(1,it)/\chi_-(1,0)$ is the characteristic function of the non positive random variable 
$S_{\tau_-}$ under $\p(\,\cdot\,|\,\tau_-<\infty)$. Since $\e(e^{\lambda S_{\tau_-}}\,|\,\tau_-<\infty)\le1$, for all $\lambda\ge0$, 
we derive from Lemma \ref{3478}  that $\chi_-(1,it)$ is determined for all $t\in\mathbb{R}$. Therefore, from (\ref{6283}), 
$\varphi(t)$ is determined for all $t\in\mathbb{R}$ and the result follows.\\ 
\end{proof}  
 
\begin{remark}\label{2699}
Distributions having negative exponential moments provide examples for which $(\ref{0236})$ is satisfied.
More specifically, assume that there is $\lambda<0$ such that $\phi(\lambda)<1$, then 
\[\p(S_n<0)\le \e(e^{\lambda S_n}\ind_{\{S_n<0\}})\le\phi(\lambda)^n\,,\]
which implies $(\ref{0236})$. The question of finding an example of a distribution with no exponential moment which satisfies 
$(\ref{0236})$ remains open.\\ 
\end{remark}

\noindent Recall that $\phi$ is a convex function on the interval $\{\alpha:\phi(\alpha)<\infty\}$. Moreover, since the support of $\mu$ is not included in $(-\infty,0)$,  $\phi(\alpha)$ is nondecreasing for $\alpha$ large enough. If
$\lambda\in\mathbb{R}$ is such that $\lambda=\inf\{\alpha:\phi(\alpha)<\infty\}$, then $\phi'(\lambda)$ will be understood as the right
derivative of $\phi$ at $\lambda$. Similarly,  if  $\lambda=\sup\{\alpha:\phi(\alpha)<\infty\}$, then $\phi'(\lambda)$ will be the left
derivative of $\phi$ at $\lambda$. 

\begin{lemma}\label{7371} 
For all $\lambda\in\mathbb{R}$ such that $\phi(\lambda)<\infty$ and $\phi'(\lambda)>0$,
\[\lim_{n\rightarrow+\infty}\e(e^{\lambda S_n}\ind_{\{S_n\ge0\}})^{1/n}=\phi(\lambda)\,.\]
\end{lemma}
\begin{proof}
Let $(S^{(\lambda)}_n)$ be a random walk with step distribution 
$\displaystyle\mu_\lambda(dx):=\frac{e^{\lambda x}}{\phi(\lambda)}\mu(dx)$. 
Since 
\[\e(S^{(\lambda)}_1)=\int_{\mathbb{R}}\frac{xe^{\lambda x}}{\phi(\lambda)}\,\mu(dx)=\frac{\phi'(\lambda)}{\phi(\lambda)}>0\,,\]
the random walk $(S^{(\lambda)}_n)$ drifts to $\infty$, so that $\lim_{n\rightarrow\infty}\p(S^{(\lambda)}_n\ge0)=1$. Then the
result follows from the identity
\[\p(S^{(\lambda)}_n\ge0)=\frac{\e(e^{\lambda S_n}\ind_{\{S_n\ge0\}})}{\phi(\lambda)^n}\,.\]
\end{proof}

The following theorem shows that distributions having some negative exponential moments less than 1 or some positive exponential 
moments bigger than 1 belong to class~$\mathscr{C}$. 

\begin{theorem}\label{2258} The knowledge of the measures $\mu^{*n}$, $n\ge1$ restricted to $[0,\infty)$ allows us
to determine if $\phi$ satisfies one of the two following conditions:
\begin{itemize}
\item[$(a)$] There exists $\lambda<0$ such that $\phi(\lambda)<1$.
\item[$(b)$] There exists  $\lambda>0$ such that $\phi(\lambda)\in(1,\infty)$.
\end{itemize}
When at least one of these two conditions holds, the measure $\mu$ 
belongs to the class $\mathscr{C}$. 
\end{theorem}
\begin{proof}  From Remark \ref{2699}, if $(\ref{0236})$ does not hold, then $(a)$ is not satisfied. Assume that $(\ref{0236})$ holds.
From Theorem \ref{4572}, the measure $\mu$ belongs to class $\mathscr{C}$ so that we can determine if $(a)$ holds.

From our data, for all $\lambda>0$ and $n\ge1$, the expression $\e(e^{\lambda S_n}\ind_{\{S_n\ge0\}})$ is known. 
Assume that there is $\lambda>0$ such that $\lim_{n\rightarrow+\infty}\e(e^{\lambda S_n}\ind_{\{S_n\ge0\}})^{1/n}>1$.
Since $\e(e^{\lambda S_n}\ind_{\{S_n\ge0\}})^{1/n}\le \phi(\lambda)$, we have actually $\phi(\lambda)>1$. Moreover,
our data clearly allows us to know if $\phi(\lambda)<\infty$. Then since $\phi(0)=1$, by convexity of the function 
$\lambda\mapsto\phi(\lambda)$, it is clear that $\phi'(\lambda)>0$, so that from Lemma \ref{7371}, 
$\lim_{n\rightarrow+\infty}\e(e^{\lambda S_n}\ind_{\{S_n\ge0\}})^{1/n}=\phi(\lambda)$. This means that from our data, we can 
determine if there is $\lambda>0$ such that $\phi(\lambda)\in(1,\infty)$ and in this case
$\lim_{n\rightarrow+\infty}\e(e^{\lambda S_n}\ind_{\{S_n\ge0\}})^{1/n}=\phi(\lambda)$, so that $\phi(\lambda)$  is known.
Moreover, from the continuity of $\phi$ 
on the set $\{x:\phi(x)<\infty\}$, there is an interval $I$ containing $\lambda$ such that for all $x\in I$, $\phi(x)\in(1,\infty)$ and
from the same reasoning as above, $\phi(x)$ is known for all $x\in I$, so that the measure $\mu$ is determined and we conclude
that it belongs to the class $\mathscr{C}$.
\end{proof}
\noindent  We can easily check that condition $(b)$ of Theorem \ref{2258} is satisfied in the two following situations:
\begin{itemize}
\item[$(b_1)$] $\phi(\lambda)<\infty$ for all $\lambda>0$.\\
\item[$(b_2)$] $\mu$ is absolutely continuous in $[0,\infty)$ and its density $f$ 
satisfies $\ln(f(x))\sim -\lambda_0 x$, as $x\rightarrow\infty$, for some $\lambda_0\in(0,\infty)$.
\end{itemize}
Indeed, in case $(b_1)$, if  $\mu\neq\delta_0$ then since $\phi$ is a nondecreasing convex function such that $\phi(0)=1$ 
and since the support of $\mu$ is not included in $(-\infty,0)$, $\lim_{\lambda\rightarrow\infty}\phi(\lambda)=\infty$. 
In case $(b_2)$ it is clear that $\lim_{\lambda\rightarrow\lambda_0-}\phi(\lambda)=\infty$. 

\subsection{Skip free distributions.}\label{skipfree} 

A distribution $\mu$ whose support is included in $\mathbb{Z}$ is said to 
be downward (resp. upward) skip free if $\mu(n)=0$ for all $n\le-2$ (resp. for all $n\ge2$).
Clearly skip free distributions possess exponential moments.  Moreover, upward skip free distributions belong to class 
$\mathscr{C}$ from Theorem \ref{2258} $(b)$ and the note following its proof. Then in this subsection we shall see that the case of 
downward skip free distributions allows us to go a little beyond the cases encompassed by Theorems \ref{4572}  
and \ref{2258}. We first need to make sure that our data allows us determine if the support of a distribution is included in $\mathbb{Z}$. 

\begin{lemma}\label{4261}
The support of the measure $\mu$ is included in $\mathbb{Z}$ if and only if the support of the measures 
$\mu^{*n}$, $n\ge1$ restricted to $[0,\infty)$ is included in $\mathbb{Z}_+$.
\end{lemma}
\begin{proof}   The direct implication is obvious. Then assume that the support of $\mu^{*n}$, $n\ge1$  restricted to 
$[0,\infty)$ is included in $\mathbb{Z}_+$, whereas the support of $\mu$ restricted to $(-\infty,0]$ is not included in 
$\mathbb{Z}_-$. 
Then there is an interval $I\subset(-\infty,0]\setminus\mathbb{Z}_-$ such that $\mu(I)>0$. 
Let $n\in\mathbb{Z}_+\setminus\{0\}$ such that $\mu(\{n\})>0$ and $h\in\mathbb{Z}_+\setminus\{0\}$ such that $hn+\inf I>0$.
Then 
\[0< \mu(I)\mu(\{n\})^h=\p(S_1\in I,S_{i+1}-S_i=n,i=1,\dots,h)\le\p(S_{n+1}\in hn+I)\,.\] 
This implies that $\mu^{*(n+1)}(hn+I)>0$, where $hn+I\subset[0,\infty)\setminus\mathbb{Z}_+$, which contradicts the 
assumption. 
\end{proof}

\begin{theorem}\label{8733}
Downward skip free distributions belong to class $\mathscr{C}$. 
\end{theorem}
\begin{proof} Let $\mu\in \mathcal{M}_1$ whose convolution powers $\mu^{*n}$, $n\ge1$ restricted to $[0,\infty)$ are known.
Then from Lemma \ref{4261} we can determine if the support of $\mu$ is included in $\mathbb{Z}$ or not. Let us assume
that it is the case. 

As already noticed at the beginning of this section, we can determine if $(S_n)$ drifts to $\infty$ or not. 
Assume first that $(S_n)$ drifts to $\infty$. 

Let us observe that under this assumption, if $\mu$ is a downward skip free distribution, then there exists $\lambda<0$ such 
that $\phi(\lambda)<1$. Indeed in this case, it is clear that $\phi(\lambda)<\infty$, for all $\lambda\le0$. Moreover $\e(S_1)$ 
exists and is positive. Since $\e(S_1)=\lim_{\lambda\uparrow0}(1-\phi(\lambda))/\lambda$, there is necessarily $\lambda<0$ 
such that $\phi(\lambda)<1$.

From Theorem \ref{2258} $(a)$, we can determine if there exists $\lambda<0$ such 
that $\phi(\lambda)<1$.  If this is not the case, then $\mu$ cannot be downward skip free. 
On the contrary, if there exists $\lambda<0$ such that $\phi(\lambda)<1$, then from Theorem \ref{2258}, $\mu$ can be determined. 
In particular, we know if $\mu$ is downward skip free or not.

Now assume that $(S_n)$ does not drift to $\infty$ and write for $n\ge0$, 
\begin{eqnarray}
\p(S_{\tau_+}>n,\tau_+<\infty)&=&\sum_{k\ge1}\p(S_1<0,\dots,S_{k-1}<0,S_k-S_{k-1}>n-S_{k-1})\nonumber\\
&=&\sum_{k\ge1}\sum_{r\le-1}\p(S_1<0,\dots,S_{k-2}<0,S_{k-1}=r)\p(S_1>n-r)\nonumber\\
&=&\sum_{r\ge1} v(r)\p(S_1>n+r)\,,\label{5202}
\end{eqnarray}  
where $v(r)=\sum_{k\ge1}\p(S_1\le0,\dots,S_{k-1}\le0,S_{k}=-r)$ is the renewal measure on $\{1,2,\dots\}$ of the (strict) 
downward ladder height process of $(S_n)$, see Chap. XII.2 in \cite{fe}.  In particular, this renewal measure 
satisfies $v(r)\le1$ for all $r\ge1$. Moreover, $(S_n)$ is downward skip free and does not drift to $\infty$, if and only if $v(r)=1$, 
for all $r\ge1$. Since it is the only unknown in Equation (\ref{5202}), we can determine if it is the case or not.  Finally, knowing that 
$\mu$ is downward skip free, we immediately determine this distribution on $\mathbb{R}$ from its knowledge on $[0,\infty)$.
\end{proof}

\noindent It appears in the proof of Theorem \ref{8733} that downward skip 
free distributions which drift to $\infty$ actually satisfy condition $(a)$ of Theorem \ref{2258}. Hence the only additional case in 
this subsection is that of downward skip free distributions which do not drift to~$\infty$.\\

We will denote by $\mathscr{E}$ the set of measures $\mu$ satisfying the assumptions of  of Theorem \ref{2258} or those of 
Theorem \ref{8733}, that is the set of measures satisfying $(a)$ or $(b)$ or downward skip free distributions.  
It will be called the {\it exponential} class. From Theorems \ref{2258} and \ref{8733}, we have $\mathscr{E}\subset\mathscr{C}$. 
Note that from Theorem \ref{4572},  we have determined a subclass of $\mathscr{C}$ which is presumably bigger than 
$\mathscr{E}$.

\section{When $\mu$ is characterized by $\mu$ and $\mu^{*2}$ on $[0,\infty)$.}\label{5673}

\subsection{Preliminary results} We will show that in many cases, the sole data of $\mu$ and $\mu*\mu$ on $[0,\infty)$
actually suffices to determine $\mu$. In this subsection, we give a theoretical condition for this to hold.\\ 

\noindent We first observe that in Conjectures C and C', there is no loss of generality in assuming that $\mu$ is absolutely
continuous on $[0,\infty)$.\\

\noindent {\bf Conjecture C''.} {\it Any absolutely continuous distribution $\mu\in\mathcal{M}_1$ whose support is not included 
in $(-\infty,0)$ is determined by its convolution powers $\mu^{*n}$, $n\ge1$ restricted to $[0,\infty)$.}

\begin{lemma}\label{3459}
Conjectures {\rm C}, {\rm C'} and {\rm C''} are equivalent.
\end{lemma}
\begin{proof} We already know from Section \ref{int} that Conjectures C and C' are equivalent. Then clearly, it suffices to prove 
that if Conjecture C'' is true, then Conjecture C' is true. 

Let $\mu,\mu_1\in\mathcal{M}_1$ be any two distributions such that the measures $\mu_1^{*n}$ and $\mu^{*n}$ agree on $[0,\infty)$, 
for  all $n\ge1$. Let $g$ be any probability density function on $[0,\infty)$, i.e.$\int_{0}^\infty g(x)\,{\rm d}x=1$ and let 
$\bar{\mu},\bar{\mu}_1\in\mathcal{M}_1$ be the absolutely continuous measures whose respective densities are $h(x)=\int_{\mathbb{R}} 
g(y-x)\,\mu(dy)$ and $h_1(x)=\int_{\mathbb{R}} g(y-x)\,\mu_1(dy)$, $x\in\mathbb{R}$.

Denoting by $g^{*n}$ the $n$-th convolution product of the function $g$ by itself, it is plain that for all $x\ge0$,
\[\bar{\mu}^{*n}[x,\infty)=\int_0^\infty\mu^{*n}[x+y,\infty)g^{*n}(y)\,dy\;\;\mbox{and}\;\;\bar{\mu}_1^{*n}[x,\infty)=
\int_0^\infty\mu_1^{*n}[x+y,\infty)g^{*n}(y)\,dy\,.\]
Therefore since the measures $\mu^{*n}$ and $\mu_1^{*n}$ agree on $[0,\infty)$, for  all $n\ge1$, the measures  $\bar{\mu}^{*n}$ 
and $\bar{\mu}_1^{*n}$ also agree on $[0,\infty)$, for  all $n\ge1$ and from the assumption that conjecture C'' is true, we conclude that 
the measures $\bar{\mu}$ and $\bar{\mu}_1$ agree on $\mathbb{R}$. Then we can identify both characteristic functions: 
\begin{eqnarray*}\int_{\mathbb{R}}e^{itx}\,\bar{\mu}(dx)&=&\int_{\mathbb{R}}e^{itx}\,\mu(dx)
\int_{0}^\infty e^{-itx}g(x)\,dx\;\;\:\mbox{and}\\
\int_{\mathbb{R}}e^{itx}\,\bar{\mu}_1(dx)&=&\int_{\mathbb{R}}e^{itx}\,\mu_1(dx)\int_{0}^\infty e^{-itx}g(x)\,dx\,.
\end{eqnarray*}
But from Lemma \ref{3024}, the characteristic function of $g$ cannot vanish identically in an interval.
This implies that $\int_{\mathbb{R}}e^{itx}\,\mu(dx)=\int_{\mathbb{R}}e^{itx}\,\mu_1(dx)$, for all $t\in\mathbb{R}$
by continuity of characteristic functions. Then we conclude that $\mu=\mu_1$ on $\mathbb{R}$, from injectivity of the 
Fourier transform. 
\end{proof}

\noindent We derive from Lemma \ref{3459} that there is no loss of generality in assuming that $\mu$ is absolutely continuous on 
$\mathbb{R}$. We will sometimes make this assumption and denote the density of $\mu$ by $f$.

\begin{lemma}\label{2363} For any probability density function, $f$ on $\mathbb{R}$ and for all $t\ge0$,
\begin{equation}\label{2925}
\int_0^\infty f(t+s)\bar{f}(s)\,ds=\frac12\left(f*f(t)-\int_0^tf(t-s)f(s)\,ds\right)\,,
\end{equation}
where $\bar{f}(s)=f(-s)$.
\end{lemma}
\begin{proof} It suffices to decompose $f*f$ as
\begin{eqnarray*}
f*f(t)&=&\int_{\mathbb{R}}f(t-s)f(s)\,ds\\
&=&\int_{-\infty}^0f(t-s)f(s)\,ds+\int_0^tf(t-s)f(s)\,ds+\int_t^\infty f(t-s)f(s)\,ds\,.
\end{eqnarray*}
Then from a change of variables, we obtain,
$\int_t^\infty f(t-s)f(s)\,ds=\int_{-\infty}^0f(t-s)f(s)\,ds=\int_0^\infty f(t+s)\bar{f}(s)\,ds$, which proves our identity.
\end{proof}

The main idea of this section is to exploit identity (\ref{2925}) in order to characterize the function $\bar{f}$ on $[0,\infty)$
(or equivalently $f $ on $(-\infty,0]$) from the sole data of $f$ and $f*f$ on
$[0,\infty)$. More specifically, assume that $f$ restricted to $[0,\infty)$ fulfills the following property:
for any two nonnegative Borel functions $g_1$ and $g_2$ defined on
$[0,\infty)$ such that $\int_0^\infty f(t+s)g_1(s)\,ds<\infty$, $\int_0^\infty f(t+s)g_2(s)\,ds<\infty$, for all $t\ge0$, the following implication is 
satisfied,
\begin{equation}\label{7326}
\int_{[0,\infty)} f(t+s)g_1(s)\,ds=\int_{[0,\infty)} f(t+s)g_2(s)\,ds\,,\;\;\mbox{for all $t\ge0$}\;\;\Rightarrow \;\;g_1\equiv g_2\,,\;\;\mbox{a.e.}
\end{equation}
Then clearly, the map $t\mapsto\int_{[0,\infty)} f(t+s)\bar{f}(s)\,ds$ characterizes $\bar{f}$ on $[0,\infty)$ and therefore from (\ref{2925}),
 $\mu$ is determined on $\mathbb{R}$ by the sole data of $\mu$ and $\mu^{*2}$ on $[0,\infty)$.

\begin{remark}\label{6922}
There are density functions $f$ which do not satisfy $(\ref{7326})$. For instance 
with $f(s)=\frac12e^{-s}$, the operator  $t\mapsto\int_{0}^\infty f(t+s)\bar{f}(s)\,ds=\frac12e^{-t}\int_{0}^\infty e^{-s}\bar{f}(s)\,ds$ provides 
a very poor information on $\bar{f}$ and certainly cannot characterize this function on $[0,\infty)$. Also if $\mu$ has a bounded 
support in $[0,\infty)$, then clearly the operator $t\mapsto\int_{0}^\infty f(t+s)\bar{f}(s)\,ds$ cannot characterize $\bar{f}$ 
outside this support. 

Note also that from $(\ref{2925})$, the knowledge of  $t\mapsto\int_{0}^\infty f(t+s)\bar{f}(s)\,ds$ and $f(t)$, for $t\ge0$ is equivalent 
to this of the functions $f$ and $f^{*2}$ on $[0,\infty)$. Therefore, in the above examples, $f$ is not even determined 
by the data of $f$ and $f^{*2}$ on $[0,\infty)$.
\end{remark}

\noindent The following proposition gives a sufficient condition for $(\ref{7326})$ to hold. 

\begin{proposition}\label{6412} Assume that $\mu$ is absolutely continuous on $\mathbb{R}$ with density $f$. 
Let us introduce the following set of functions defined on $[0,\infty)$,
\[\mathcal{H}:=\left\{\sum_{k=1}^n\alpha_kf(t_k+\cdot):n\ge1,\alpha_k\in\mathbb{R},\,t_k\ge0\right\}\,.\]
If the restriction of $f$ to $[0,\infty)$ belongs to $L^\infty([0,\infty))$ and if
$\mathcal{H}$ is dense in $L^\infty([0,\infty))$, then for any $g\in L^1([0,\infty))$, the following implication is satisfied,
\begin{equation}\label{7325}
\int_{0}^\infty f(t+s)g(s)\,ds=0,\;\;\mbox{for all $t\ge0$}\;\;\Rightarrow \;\;g\equiv0,\;\;\mbox{a.e.}
\end{equation}
When $(\ref{7325})$ holds, the measure $\mu$ is determined on $\mathbb{R}$ by $\mu$ and $\mu^{*2}$ on $[0,\infty)$. 
\end{proposition}
\begin{proof} If $\int_{0}^\infty f(t+s)g(s)\,ds=0$, for all $t\ge0$, then clearly, since $\mathcal{H}$ is dense in $L^\infty([0,\infty))$, 
$\int_{0}^\infty h(s)g(s)\,ds=0$ for all $h\in L^\infty([0,\infty))$ and this implies that $g\equiv 0$, a.e.

Assume now that the restriction of the measures $\mu$ and $\mu^{*2}$ are known on $[0,\infty)$. Recall the notation  
$\bar{f}$ from Lemma \ref{2363} and observe that the right hand side of identity (\ref{2925}) is known for all $t\ge0$. 
From (\ref{7325}) this determines $\bar{f}$ on $[0,\infty)$ and the measure $\mu$ is determined.
\end{proof}
\noindent Unfortunately we do not know any example of function satisfying the condition of Proposition~$\ref{6412}$ and finding a 
simple criterion on $f$ for it to satisfy this condition remains an open problem. More specifically, we may wonder if the converse of
Proposition~$\ref{6412}$ holds, that is if assertion $(\ref{7325})$ implies that $\mathcal{H}$ is dense in $L^\infty([0,\infty))$.
The latter problem can be compared with Wiener's approximation theorem which asserts that for a function $f$ in 
$L^1(\mathbb{R})$ the set $\mathcal{H}$ $($thought as a set of functions defined on $\mathbb{R}$$)$ is dense in 
$L^1(\mathbb{R})$ if and only if the Fourier transform of $f$ does not vanish, see \cite{ru}.\\

\noindent In Subsection \ref{7235}, we give a class of density functions such that (\ref{7326}) holds and in Subsection \ref{3468},
 we give a class of density functions which are bounded on $[0,\infty)$ and such that (\ref{7325}) holds.

\subsection{The completely monotone class}\label{monotone}\label{7235}

In this subsection, we assume that $\mu$ is absolutely continuous with respect to the Lebesgue measure on $\mathbb{R}$ and we
denote by $f$ its density.\\

\noindent We will show that if $f$ restricted to $(a,\infty)$, for some $a\ge0$ is a completely monotone function satisfying some mild 
additional assumption, then $\mu$ is characterized from $\mu$ and $\mu^{*2}$ on $[0,\infty)$. Let us first recall that from Bernstein 
Theorem, the function $f$ is completely monotone on $(a,\infty)$, for $a\ge0$, if and only if there is a positive Borel measure $\nu$ 
on $(0,\infty)$ such that for all $t>a$, 
\begin{equation}\label{1842}
f(t)=\int_0^\infty e^{-ut}\nu(du)\,.
\end{equation}

\begin{theorem}\label{9524} Assume that there is $a\ge0$ such that the restriction of $f$ to $(a,\infty)$ is completely monotone. 
Assume moreover  that the support of the measure $\nu$ in $(\ref{1842})$ contains an increasing sequence $(a_n)$ such that 
$\sum_n a_n^{-1}=+\infty$. Then $(\ref{7326})$ holds and the measure $\mu$ is characterized by the restriction of $\mu$ and 
$\mu^{*2}$ to $[0,\infty)$. In particular, $\mu$ belongs to class $\mathscr{C}$.
\end{theorem}
\begin{proof} Let $g$ be any nonnegative Borel function defined on $[0,\infty)$ such that $\int_0^{\infty}f(t+s)g(s)\,ds<\infty$, for all 
$t>a$. Then from Fubini's Theorem, for all $t>a$, 
\begin{eqnarray*} 
\int_0^{\infty}f(t+s)g(s)\,ds&=&\int_0^{\infty}\int_0^\infty e^{-u(t+s)}\nu(du)g(s)\,ds\\
&=&\int_0^{\infty}e^{- ut}\int_0^\infty e^{-us}g(s)\,ds\,\nu(du)\,.
\end{eqnarray*}
This expression is the Laplace transform of the measure $\theta(du):=\int_0^\infty e^{-us}g(s)\,ds\,\nu(du)$. 
The knowledge of this Laplace transform for all $t>a$ characterizes the measure $\theta(du)$ so that a version of the density 
function $u\mapsto \int_0^\infty e^{-us}g(s)\,ds$ is known on a Borel set $B\subset (0,\infty)$ such that $\nu(B^c)=0$. Since this 
density function is continuous, it is known on $\overline{B}$ and hence it is known everywhere on the support of $\nu$. Therefore, 
from the assumption on $\nu$, we can find an increasing sequence $(a_n)$ such that $\sum_n a_n^{-1}=+\infty$ and such that the 
Laplace transform $\int_0^\infty e^{-a_ns}g(s)\,ds$ of the function $g$ at $a_n$ is known for each $n$. 
From a result in \cite{fe1}, this is enough to determine the function $g$ and (\ref{7326}) holds, see also \cite{fe}, p.430. 

Then we derive from (\ref{7326}) and Lemma \ref{2363} that the function $\bar{f}$ is determined on $[0,\infty)$ from the restriction 
on $[0,\infty)$ of $\mu$ and $\mu^{*2}$. Therefore the measure $\mu$ is determined. 
\end{proof}

\noindent We will denote by $\mathscr{M}$ the set of absolutely continuous measures $\mu$ whose density $f$ satisfies the 
assumption of Theorem \ref{9524}. This class will be called the completely monotone class. Theorem \ref{9524} shows that 
$\mathscr{M}\subset\mathscr{C}$.

\begin{remark}
Note that since $f$ is a density function, the measure $\nu$ in $(\ref{1842})$ should 
also satisfy 
\begin{equation}\label{6759}
\int_0^\infty f(t)\,dt=\int_0^\infty \frac{e^{-au}}u\nu(du)\le 1\,.
\end{equation}
\end{remark}

\begin{remark}\label{8644} Clearly class $\mathscr{E}$ is not included in class $\mathscr{M}$. Moreover, it is easy to find an 
example of a measure in class $\mathscr{M}$ which does not belong to class $\mathscr{E}$. 
Indeed, we readily check that whenever the support $S$ of $\nu$ is such that $S\cap(0,\varepsilon)\neq\emptyset$, 
for all $\varepsilon>0$, then $\mu$ has no positive exponential moments, so that it cannot belong to class 
$\mathscr{E}$. Let us take for instance the measure $\nu(du)$ on $[0,\infty)$ with density 
$u^{2}\ind_{\{u\in[0,1]\}}+\frac14u^{-1}\ind_{\{u\in(1,\infty)\}}$. 
This measure satisfies the assumption of Theorem $\ref{9524}$, 
for $a=0$. \\
\end{remark}

\begin{remark}
Theorem $\ref{9524}$ excludes completely monotone functions of the type $f(t)=\sum_{k=1}^n\alpha_ke^{-\beta_k t}$, $t>a$, for 
some $\alpha_k,\beta_k>0$ and some finite $n$ since in this case the measure $\nu(du)=\sum_{k=1}^n\alpha_k\delta_{\beta_k}(du)$ 
does not satisfy the condition required by this theorem. It also excludes functions $f$ whose support is bounded since completely 
monotone functions on $(a,\infty)$ are analytic on $(a,\infty)$.   This remark is consistent with Remark $\ref{6922}$. 
\end{remark}

\subsection{The analytic class}\label{analytic}\label{3468}

Let us assume again that $\mu$ has density $f$ on $\mathbb{R}$. We will now exploit the same kind of arguments as in the previous 
subsection by assuming that $f$ is the Fourier transform of some complex valued function.  

\begin{theorem}\label{3720} Assume that there is a complex valued function $k$ such 
that for all $t\ge0$,
\begin{equation}\label{8179}
f(t)=\int_{\mathbb{R}}e^{iut}k(u)\,du\,.
\end{equation}
Assume moreover that
\begin{itemize}
\item[1.] the absolute moments $M_n=\int_{\mathbb{R}}|u|^n|k(u)|\,du$, $n\ge0$ are finite and satisfy
\[\limsup_n\frac{M_n}{n!}<\infty\,,\]
\item[2.] the function $k$ does not vanish on any interval of $\mathbb{R}$.
\end{itemize}
Then $f$ is bounded on $[0,\infty)$ and $(\ref{7325})$ holds. In particular, the measure $\mu$ is characterized by 
the restrictions of $\mu$ and $\mu^{*2}$ to $[0,\infty)$ and $\mu$ belongs to class $\mathscr{C}$. Moreover $f$ admits an analytic 
continuation on $\mathbb{R}$.
\end{theorem}
\begin{proof} The fact that $f$ is bounded follows directly from (\ref{8179}) and $1.$ Let $g\in L^1([0,\infty))$, 
then from (\ref{8179}) and Fubini's theorem, we can write for all $t\ge0$, 
\begin{eqnarray*}
\int_{0}^\infty f(t+s)g(s)\,ds&=&\int_0^{\infty}\int_{\mathbb{R}}e^{iu(t+s)}k(u)\,du\,
g(s)\,ds\\
&=&\int_{\mathbb{R}}e^{iut}\int_0^{\infty}e^{ius}g(s)\,ds\,k(u)\,du\,.
\end{eqnarray*}
Assume that this expression vanishes for all $t\ge0$ and set $\varphi(u)=\int_0^{\infty}e^{ius}g(s)\,ds$. 
This means that the Fourier transform 
\[\Psi(t):=\int_{\mathbb{R}}e^{iut}\varphi(u)\,k(u)\,du\,,\] 
of the function $u\mapsto \varphi(u)k(u)$, $u\in\mathbb{R}$ vanishes for all $t\ge0$. Then let us show that under our assumptions, 
the function $\Psi$ is analytic on the whole real axis. First note that since $|\varphi(u)|\le\|g\|_{L^1}$, $u\in\mathbb{R}$ and since all the 
moments $M_n$ are finite, then $\Psi$ is infinitely differentiable on $\mathbb{R}$ and 
\begin{equation}\label{4168}
\Psi^{(n)}(t)=\int_{\mathbb{R}}(iu)^ne^{iut}\varphi(u)\,k(u)\,du\,,\;\;\;t\in\mathbb{R}\,.
\end{equation}
Then notice that for all $t,u,x\in\mathbb{R}$,
\[\left|e^{iux}\left(e^{itu}-1+\frac{itu}{1!}-\dots-\frac{(itu)^{n-1}}{(n-1)!}\right)\right|\le \frac{|tu|^n}{n!}\,.\]
We derive from this inequality and (\ref{4168}) that 
\begin{equation}\label{4199}
\left|\Psi(x+t)-\Psi(x)-\frac t{1!}\Psi'(x)-\dots-\frac{t^{n-1}}{(n-1)!}\Psi^{(n-1)}(x)\right|\le 
\frac{M_n}{n!}|t|^n\,.
\end{equation}
Set $c=\limsup_nM_n/n!$, then from Stirling's formula, for $|t|<1/(3c)$, the righthand side of (\ref{4199}) 
tends to 0, as $n\rightarrow+\infty$, hence the Taylor series of $\Psi$ converges in some interval around $x$, 
for all $x\in\mathbb{R}$. It follows that $\Psi$ is analytic on $\mathbb{R}$. As a consequence, 
$\Psi$  is determined by its expression on the positive half line. Hence the Fourier transform of the continuous function 
$u\mapsto\varphi(u)k(u)$ vanishes on $\mathbb{R}$, which means that this function vanishes a.e. on $\mathbb{R}$. 
Since $k$ does not vanish on any interval of $\mathbb{R}$ and $\varphi$ is continuous, it implies that $\varphi(u)=0$, 
for all $u\in\mathbb{R}$ and we conclude that $g(t)=0$, for almost every $t\in[0,\infty)$. We have proved that 
$(\ref{7325})$ holds and from the second part of Proposition \ref{6412}, $\mu$ is characterized by the restriction of $\mu$ 
and $\mu^{*2}$ on $[0,\infty)$.    

We have proved above that $\Psi(t)=\int_{\mathbb{R}}e^{iut}\varphi(u)\,k(u)\,du$ is analytic on $\mathbb{R}$. 
It follows from the same arguments that the continuation of $f$ on $\mathbb{R}$ which is defined in a natural way by 
$f(t)=\int_{\mathbb{R}}e^{iut}k(u)\,du$, $t\in\mathbb{R}$ is analytic, which proves the last assertion of the theorem. 
\end{proof}

We will denote by $\mathscr{A}$ the class of distributions which satisfy the assumptions of Theorem \ref{3720}. It will be 
called the analytic class. From Theorem \ref{3720}, $\mathscr{A}\subset\mathscr{C}$.\\  

It is very easy to construct examples of distributions in class $\mathscr{A}$, simply by choosing any symmetric function 
$k$ which satisfies assumptions $1.$ and $2.$ in Theorem \ref{3720}. Let us consider for instance $k(u)=e^{-|u|}$,  for 
$u\in\mathbb{R}$. Then 
\[f(t)=\int_{\mathbb{R}}e^{iut}e^{-|u|}\,du=\frac{1}{2(1+t^2)}\,,\;\;\;t\ge0\,.\]
Any extension on $\mathbb{R}$ of this function into a density function determines a distribution of $\mathscr{A}$ which does 
not belong to classes $\mathscr{E}$ and $\mathscr{M}$. Conversely, none of the classes $\mathscr{E}$ and $\mathscr{M}$ 
is included in $\mathscr{A}$. It is straightforward for $\mathscr{E}$. 
Then let us consider 
\[f(t)=\frac14\left(\int_0^1 e^{-ut}\,u^{1/2}du+\int_1^\infty e^{-ut}\,u^{-1}du\right)\,,\;\;\;t>0\,.\] The measure 
$\nu(du)=\left(u^{1/2}\ind_{[0,1]}(u)+u^{-1}\ind_{[1,\infty)}(u)\right)\,du$ satisfies the 
condition of Theorem \ref{9524} so that $\mu\in\mathscr{M}$ (here we  choose $a=0$). However, 
$\lim_{t\rightarrow0+}f(t)=\infty$, hence it does not admit an analytic continuation on $\mathbb{R}$, so that $f$ does not belong 
to class $\mathscr{A}$, from Theorem \ref{3720}. (Note also that since the support of $\nu$ intersects any interval
$(0,\varepsilon)$, $\varepsilon>0$, the measure $\mu$ has no positive exponential moments, see Remark \ref{8644}.)\\

Here is a consequence of Theorem \ref{3720} on stable distributions. 

\begin{corollary}\label{2885}
Let $S$ be a stable distribution on $\mathbb{R}$ with index $\alpha\in[1,2]$, whose support is not 
included in $(-\infty,0)$. If a measure $\mu$ satisfies $\mu=S$ and $\mu*\mu=S*S$ on $[0,\infty)$ then $\mu=S$ on 
$\mathbb{R}$.
\end{corollary}
\begin{proof} Let $f$ be the density of $\mu$ in $[0,\infty)$. From the expression of the characteristic exponent of stable 
distributions and Fourier inverse transform, for all $t\ge0$, 
\[f(t)=\frac1{2\pi}\int_{\mathbb{R}}e^{-itu}e^{-c|u|^\alpha(1-i\beta\mbox{\rm\tiny sgn}(u)\tan(\pi\alpha/2))}\,du\,,\]
where $\beta\in[-1,1]$ and $c$ is some positive constant. Then we can easily check that if $\alpha\in[1,2]$, $f$ satisifes the 
conditions of Theorem \ref{3720}  and the result follows. 
\end{proof} 

\noindent Corollary \ref{2885} should be compared with a result from Rossberg and Jesiak \cite{rj} which asserts that if 
$F_1$ and $F_2$ are the distribution functions of two stable distributions and if $F_1(x)=F_2(x)$ for $x$ belonging to
a set which contains at least three accumulation points, $l_1$, $l_2$ and $l_3$ such that $F_i(l_j)\neq 0,1$ for $i=1,2$
and $j=1,2,3$, then $F_1\equiv F_2$, on $\mathbb{R}$. We stress that in Corollary \ref{2885} it is not even known a priori 
that $\mu$ is an infinitely divisible distribution. It is actually conjectured in \cite{rj} that if $F$ is the distribution
function of an infinitely divisible distribution and satisfies $F(x)=S(x)$, for all $x\ge0$, then $F\equiv S$ on $\mathbb{R}$.
This problem is solved (and proved to be true) only for $\alpha=2$. 

\subsection{A class of discrete distributions}\label{6288}

In this section, we present the discrete counterpart of class $\mathscr{M}$. More specifically, we will 
consider distributions whose support is included in $\mathbb{Z}$.\\

First we need the following equivalent of Lemma \ref{2363} for discrete distributions. Its proof is straightforward
so we omit it. 

\begin{lemma}\label{9742}  Let $(q_n)_{n\in\mathbb{Z}}$ be any probability on $\mathbb{Z}$. Define 
$q^{*2}_n=\sum_{k\in\mathbb{Z}}q_{n-k}q_k$. 
Then for all $n\ge1$ 
\begin{equation}\label{5821}
\sum_{k=-\infty}^0q_{n-k}q_k=\frac12\left(q^{*2}_n-\sum_{k=1}^{n-1}q_{n-k}q_k\right)\,,
\end{equation}
where we set $\sum_{k=1}^{n-1}q_{n-k}q_k=0$ if $n=1$. 
\end{lemma}

A sequence $(a_k)_{k\ge0}$ of nonnegative real numbers is called completely monotone if for all $k\ge0$ and $n\ge1$,
\[\Delta^na_k:=\Delta^{n-1}a_k-\Delta^{n-1}a_{k+1}\ge0\,,\]
where $\Delta^0a=a$. A result from Hausdorff asserts that $(a_k)_{k\ge0}$ is completely monotone if and only if there is a 
finite measure $\nu$ on $[0,1]$ such that for all $k\ge0$,
\begin{equation}\label{3936}
a_k=\int_0^1t^k\,\nu(dt)\,.
\end{equation}

\noindent Let us set $\mu(\{n\})=\mu_n$, for $n\in\mathbb{Z}$. Then by assuming that $(\mu_n)_{n\ge0}$ is completely 
monotone, we obtain a new class of distributions satisfying conjecture $\mathscr{C}$ as shows the following theorem.

\begin{theorem}\label{9756}  Assume that $(\mu_n)_{n\ge0}$ is completely monotone and that the support of the measure 
$\nu$ in the representation $(\ref{3936})$ contains a decreasing sequence $(c_n)$ such that $\sum (-\ln c_n)^{-1}=\infty$. 
Then the data of $\mu$ and $\mu^{*2}$ restricted to $[0,\infty)$ allows us to determine if the support of $\mu$ is included in 
$\mathbb{Z}$. Assume that this is the case, then the measure $\mu$ is determined by the restriction of $\mu$ and $\mu^{*2}$ 
to $[0,\infty)$. In particular $\mu$ belongs to class $\mathscr{C}$. 
\end{theorem}
\begin{proof} First let us observe that we can derive from the data of 
of $\mu$ and $\mu^{*2}$ restricted to $[0,\infty)$ that the support of $\mu$ is included in $\mathbb{Z}$. Indeed, assume 
that  the support of $\mu$ restricted to $(-\infty,0]$ is not included in $\mathbb{Z}_-$. Then there is an interval
$I\subset(-\infty,0]\setminus\mathbb{Z}_-$ such that $\mu(I)>0$. From (\ref{3936}), $\mu(\{n\})>0$ for all $n\ge1$. 
Let $n\in\mathbb{Z}_+\setminus\{0\}$ such that $n+\inf I>0$, then 
\[0< \mu(I)\mu(\{n\})=\p(S_1\in I,S_{2}-S_1=n)\le\p(S_{2}\in n+I)\,.\] 
This implies that $\mu^{*2}(n+I)>0$, where $n+I\subset[0,\infty)\setminus\mathbb{Z}_+$, which contradicts the assumption. 

From the Haussdorff representation recalled above, there is a unique finite measure $\nu$ on [0,1] such that 
$\mu_k=\int_0^1t^k\,\nu(dt)$. Using this representation and Fubini's theorem, we can write for all $n\ge0$,
\[\sum_{k=0}^\infty \mu_{-k}\mu_{n+k}=\int_0^1t^n\left(\sum_{k=0}^\infty \mu_{-k}t^k\right)\,\nu(dt)\,.\]
From (\ref{5821}) in Lemma \ref{9742} 
applied to $(\mu_k)$, we derive that this expression is determined from the knowledge of $\mu$ and $\mu^{*2}$ 
on  $[0,\infty)$.  This means that we know the moments of the measure $\left(\sum_{k=0}^\infty \mu_{-k}t^k\right)\cdot\nu(dt)$. 
This measure is finite and its support is included in $[0,1]$, hence it is determined by its moments. Then we know 
the generating function $t\mapsto \sum_{k=0}^\infty \mu_{-k}t^k$ of the sequence $(\mu_{-k})_{k\ge0}$ on the support of $\nu$, 
since this function is continuous. From the assumption, we know this generating function on a sequence $(c_n)$ such that 
$\sum (-\ln c_n)^{-1}=\infty$. This is enough to determine the sequence $(\mu_{-k})_{k\ge0}$, from \cite{fe1}.

We conclude that the measures $\mu$ and $\mu^{*2}$ restricted to $[0,\infty)$ allow us to determine $\mu$ on $\mathbb{Z}$.
\end{proof}

The set of measures satisfying the assumptions  of Theorem \ref{4261} will be called the discrete monotone class and will be denoted by 
$\mathscr{M}_d$. Theorem \ref{4261} shows that $\mathscr{M}_d\subset \mathscr{C}$. Moreover, it is clear that none of the classes 
$\mathscr{E}$, $\mathscr{M}$ and  $\mathscr{A}$ is included in $\mathscr{M}_d$ and that these classes do not contain  $\mathscr{M}_d$.

\section{When $\mu$ is infinitely divisible}\label{2482}

The aim of this section is to present a problem equivalent to conjecture  $\mathscr{C}$ in the framework of infinitely divisible distributions. 
When $\mu$ is infinitely divisible, the Wiener-Hopf factorization  can be understood in two different ways: we can either factorize the 
characteristic function $\varphi$ as in (\ref{6283}), or we can factorize the characteristic exponent $\psi$, which is defined by 
\[\varphi(t)=e^{\psi(t)}\,,\;\;\;t\in\mathbb{R}\,.\]
Then let us recall the Wiener-Hopf factorization in the latter context.
Let $(X_t,\,t\ge0)$ be a real L\'evy process issued from 0 under the probability $\p$ and such that $X_1$ has law $\mu$ 
under this probability, that is $\e(e^{iuX_t})=e^{-t\psi(u)}$, for all $t\ge0$. The  characteristic exponent of $\mu$ is
given explicitly according to the L\'evy-Khintchine formula by
\[\psi(u)=iau+\frac{\sigma^2}2u^2+\int_{\mathbb{R}\setminus\{0\}}(1-e^{iux}+iux1_{\{|x|\le1\}})\,\Pi(dx)\,,\]
where $a\in\mathbb{R}$, $\sigma\ge0$ and $\Pi$ is a measure on $\mathbb{R}\setminus\{0\}$, such that
$\int(x^2\wedge 1)\,\Pi(dx)<\infty$. Then the Wiener-Hopf factorization 
of $\psi$ has the following form:
\begin{equation}\label{6228}
s+\psi(u)=\kappa_+(s,-iu)\kappa_-(s,iu)\,,\;\;\;u\in\mathbb{R}\,,\;\;s\ge0\,,
\end{equation}
where $\kappa_+$ and $\kappa_-$ are the Laplace exponents of the upward and downward ladder processes
$(\tau^+,H^+)$ and $(\tau^-,H^-)$ of $X$, that is $\e(e^{-\alpha\tau_t^{+/-}-i\beta H_t^{+/-}})=e^{-t\kappa_{+/-}(\alpha,\beta)}$. These
exponents are given explicitly for $\alpha,\beta\ge0$ by the identities,
\begin{eqnarray}
\kappa_-(\alpha,\beta)&=&k_-\exp\left(\int_0^\infty\int_{(-\infty,0)}(e^{-t}-e^{-\alpha t-\beta x})\frac1t\p(X_t\in dx)\,dt\right)\\
\kappa_+(\alpha,\beta)&=&k_+\exp\left(\int_0^\infty\int_{[0,\infty)}(e^{-t}-e^{-\alpha t-\beta x})\frac1t\p(X_t\in dx)\,dt\right)\,,\label{2324}
\end{eqnarray}
where $k_-$ and $k_+$ are positive constants depending on the normalization of the local times at the infimum and at the supremum of $X$.
The joint law of $(\tau_1^{+},H_1^{+})$ is the continuous time counterpart of the joint law $(\tau_{+},S_{\tau_{+}})$ 
defined in Section \ref{int}, in the setup of random walks. We refer to Chap. VI of \cite{be}, Chap. IV of \cite{ky} or Chap. of \cite{do} for 
complete definitions of these notions. Note that our formulation of the Wiener-Hopf factorization (\ref{6228}) includes compound Poisson 
processes since expression (\ref{2324}) takes account of a possible mass at 0 for the measure $\p(X_t\in dx)$. This slight extension can be 
derived from p. 24 and 25 of \cite{vi}, see also the end of Section 6.4, p.183 in \cite{ky}.

Set $\overline{\Pi}(t)=\Pi(t,\infty)$, $t>0$ and let $\mu_t$ be the law of $X_t$.

\begin{lemma}\label{4237}
The data of $\kappa_+$ is equivalent to that of $\mu_t$ on $[0,\infty)$, for all $t\ge0$.
Moreover, the knowledge of $\kappa_+$ allows us to determine the function $\overline{\Pi}(t)$, $t>0$.
\end{lemma}
\begin{proof}
From identity
\begin{equation}\label{1263}
\frac1t\p(X_t\in dx)\,dt=\int_0^\infty\p(\tau_u\in dt,H_u\in dx)\,\frac{du}u\,,\;\;x\ge0,t>0\,,
\end{equation}
which can be found in Section 5.2 of \cite{do}, we see that the law of $X_t$ in $[0,\infty)$, for all $t\ge0$ is determined
by the law of $(\tau,H)$ and hence by $\kappa^+$. (Note that equation (\ref{1263}) is also valid for compound Poisson processes.) 
Conversely, it follows directly from formula (\ref{2324}) that $\kappa^+$ is determined by the data of the measure $\mu_t$ on $[0,\infty)$, 
for all $t\ge0$.

The second assertion is a consequence of the first one and Exercise~1 of chap.I in \cite{be}, which asserts that the
family of measures $\frac1t\p(X_t\in dx)$ converges vaguely toward $\Pi$, as $t\rightarrow0$.
\end{proof}

\noindent The above lemma enables us to make the connection between the two Wiener-Hopf factorizations 
(\ref{6283}) and (\ref{6228}). Let us state it more specifically in the following proposition. 

\begin{proposition}\label{3783}
The Wiener-Hopf factor $\kappa_+$ allows us to determine the Wiener-Hopf factor $\chi_+$.
\end{proposition}
\begin{proof}
The result is straightforward from Lemma \ref{4237}. Indeed, knowing $\kappa_+$ we can determine $\mu_n=\mu^{*n}$ restricted 
to $[0,\infty)$, for all $n\ge1$ and from Section \ref{int} that this data is equivalent to that of $\chi_+$. 
\end{proof}

\begin{definition}
We will denote by $\mathscr{C}_i$ the class of infinitely divisible distributions $\mu$ which are determined by the data of their 
upward Wiener-Hopf factor $\kappa_+(s,t)$, for $s,t\ge0$ or equivalently by the data of the measures $\mu_t$, 
$t>0$ restricted to $[0,\infty)$. 
\end{definition}

\noindent Let us denote by $\mathscr{I}$ the set of infinitely divisible distributions. Then it is straightforward 
that $\mathscr{I}\cap\mathscr{C}\subset\mathscr{C}_i$. In particular if Conjecture C is true, then $\mathscr{C}_i=\mathscr{I}$.
It was proved in Chapter 4 of \cite{vi} that infinitely divisible distributions having some exponential moments belong to class $\mathscr{C}_i$, 
which is a consequence of our results. The latter work uses a different technique based on the analytical continuation of the Wiener-Hopf 
factors $\kappa_+$ and $\kappa_-$. \\

Let $k_-,\delta_-,\gamma_-$ and $k_+,\delta_+,\gamma_+$ be the killing rate, the drift and the L\'evy measure of 
the subordinators $H_-$ and $H_+$, respectively and let us set 
$\bar{\gamma}_+(x)=\gamma_+(x,\infty)$ and $\bar{\gamma}_-(x)=\gamma_-(x,\infty)$. Let also
 $U_-$ be the renewal measure of the  downward ladder height process $H^-$, that is $U_-(dx)=\int_0^\infty\p(H_t^-\in dx)$.

\begin{theorem}\label{3920}  Assume that the function $t\in(a,\infty)\mapsto \overline{\Pi}(t)$ is completely monotone, for some $a\ge0$,
that is there exists a Borel measure $\nu$ on $(0,\infty)$ such that for all $t>a$,
\begin{equation}\label{9422}
\overline{\Pi}(t)=\int_0^\infty e^{-ut}\,\nu(du)\,.
\end{equation}
Assume moreover that the support of $\nu$ contains an increasing sequence $(a_n)$ such that $\sum_{n}a_n^{-1}=+\infty$.
Then the measure $\mu$ belongs to the class $\mathscr{C}_i$.
\end{theorem}
\begin{proof} 
The proof relies on Vigon's \'equation amicale invers\'ee, see \cite{vi}, p.71, or (5.3.4) p.44 in \cite{do} which can be written as
\begin{equation}\label{4581}
\bar{\gamma}_+(x)=\int_{[0,\infty)}U_-(dy)\overline{\Pi}(x+y)\,,\;\;\;x>0\,.
\end{equation}
Note that (\ref{4581}) is analoguous to (\ref{5202}). From Lemma \ref{4237}, given $\kappa_+$, we know both $\bar{\gamma}_+(x)$, 
for $x>0$ and $\overline{\Pi}(t)$, for $t>0$. Then we will show that under our assumption, equation (\ref{4581}) allows us to determine 
the renewal measure $U_-(dy)$, so that the law of $X$ will be entirely determined, thanks to the relation:
\begin{eqnarray}
\hat{U}_-(z)&=&\int_\mathbb{R_+}e^{-yz}U_-(dy)\nonumber\\
&=&\frac1{\kappa_-(0,z)}\label{3471}\,,\;\;\;z>0\,,
\end{eqnarray}
and the Wiener-Hopf factorization (\ref{6228}).

From (\ref{9422}), (\ref{4581}) and Fubini's Theorem, we can write for all $x>0$,
\begin{eqnarray}
\bar{\gamma}_+(x)&=&\int_{[0,\infty)}U_-(dy)\int_0^\infty e^{-(x+y)z}\,\nu(dz)\nonumber\\
&=&\int_0^\infty e^{-xz}\hat{U}_-(z)\,\nu(dz)\,.\label{4925}
\end{eqnarray}
Then the left hand side of equation (\ref{4925}) determines the measure $\hat{U}_-(z)\,\nu(dz)$.
Since $z\mapsto\hat{U}_-(z)$ is a continuous function, then it is determined on the support of $\nu$. From our assumption 
on this support and \cite{fe1} we derive that $\hat{U}_-$ (and hence $\kappa^-(0,z)=-\log \mathbb{E}(e^{-z H^-_1})$) 
for $z>0$, is determined.
\end{proof}

Note that an analogous result to Theorem \ref{3920}  holds when $\mu$ has support in $\mathbb{Z}$ and the 
sequence $\Pi(n)$, $n\ge1$ 
satisfies the same assumptions as $(\mu_n)_{n\ge0}$ in Theorem \ref{9756}. One may also wonder if an assumption such as 
(\ref{8179}) for $\overline{\Pi}(t)$ would lead to a similar result to Theorem \ref{3720}. However, in order to use the same 
argument as in the proof of this theorem together with equation (\ref{4581}), we need $\hat{U}_-(z)$ to be bounded, which is 
not the case in general.\\

\begin{remark}
As already observed above, the class $\mathscr{C}_i$ contains at least all probability measures in the set
$\mathscr{I}\cap(\mathscr{E}\cup\mathscr{M}\cup\mathscr{A}\cup\mathscr{M}_d)$ but Theorem $\ref{3920}$ shows that there 
are other distributions in $\mathscr{C}_i$. Indeed, it is easy to construct an example of a compound Poisson process $(X_t,\,t\ge0)$
with intensity $1$, whose L\'evy measure $\Pi$ satisfies conditions of Theorem $\ref{3920}$ but such that the law 
$\mu(dx)=e^{-1}\sum_{n\ge0}\Pi^{*n}(dx)/{n!}$ of $X_1$ does not belong to any of the classes 
$\mathscr{E}$, $\mathscr{M}$, $\mathscr{A}$ and $\mathscr{M}_d$.
\end{remark}

An infinitely divisible distribution is said to be downward skip free (respectively upward skip free) if the support of the measure $\Pi$ is 
included in $(-\infty,0]$ (respectively in $[0,\infty)$). Upward skip free distributions clearly belong to class $\mathscr{C}_i$ from 
the Wiener-Hopf factorization (\ref{6228}). Then here is a counterpart of Theorem \ref{8733}. 

\begin{theorem}\label{2455}  
Downward skip free infinitely divisible distributions belong to the class $\mathscr{C}_i$.
\end{theorem}
\begin{proof} The proof relies on Vigon's \'equation amicale, p.71 in \cite{vi}. See also equation (5.3.3), p.44 in \cite{do}. 
If $\delta_->0$, then from \cite{vi}, the L\'evy measure $\gamma_+$ is absolutely continuous and we will denote by $\gamma_+(x)$ its 
density. Then Vigon's \'equation amicale can be written as 
\[\overline{\Pi}(x)=\int_0^\infty\gamma_+(x+du)\bar{\gamma}_-(u)+\delta_-\gamma_+(x)+k_-\bar{\gamma}_+(x)\,,\;\;\;x>0\,.\]
It is plain that in the right hand side, the term $\int_0^\infty\gamma_+(x+du)\bar{\gamma}_-(u)$ is identically 0 if and only if the L\'evy 
process $X$ is spectrally positive, that is $\mu$ is downward skip free. Moreover, (\ref{6228}) for $u=0$ entails that the knowledge of 
$\kappa_+$ implies that of $\kappa_-(s,0)$, for all $s\ge0$. In particular, we know the killing rate of the subordinator $(\tau_t^-,t\ge0)$, 
and this killing rate is the same as this of $(H_t^-,t\ge0)$, that is $k_-$.  
Then we conclude that $X$ is spectrally positive if and only if there is a constant $\delta_-$ such that 
\[\overline{\Pi}(x)=\delta_-\gamma_+(x)+k_-\bar{\gamma}_+(x)\,,\;\;\;x>0\,,\]
and this can be determined, since from our data, we know $k_-$, $\overline{\Pi}(x)$ and $\bar{\gamma}_+(x)$, for $x>0$.
\end{proof}

\section{More classes of distributions}\label{2450}

In the previous sections, we have highlighted the subclasses $\mathscr{E}$, $\mathscr{M}$, $\mathscr{M}_d$ and $\mathscr{A}$ of 
$\mathscr{C}$ and proved that these sets of distributions are distinct  from each other. More specifically, none of them is included into 
another one. Then the aim of this section is to show that some of these classes can be substantially enlarged through simple 
arguments.\\

Actually for most of the subclasses investigated in this paper, we imposed conditions bearing only on $\mu$ restricted to $[0,\infty)$, 
but one is also allowed to make assumptions on $\mu^{*n}$ restricted to $[0,\infty)$. In order to move in this direction, let us mention 
the following straightforward extension of results of Section \ref{5673}.

\begin{proposition}\label{3555} Let $\mu\in\mathcal{M}_1$ be absolutely continuous with density $f$. 
If there is $n\ge1$ such that the density function $f^{*n}$ satisfies the same conditions as $f$ in Theorems $\ref{9524}$ or in
Theorem $\ref{3720}$, then $\mu$ is determined by $\mu^{*n}$ and $\mu^{*2n}$ restricted to $[0,\infty)$. In particular $\mu$ 
belongs to class $\mathscr{C}$. 
\end{proposition}
\noindent It is plain that an analogous extension of Theorem \ref{9756} is satisfied.
Then here is a more powerful result  allowing us to extend our classes of distributions.

\begin{theorem}\label{8455}  
Let $\mu\in\mathcal{M}_1$. If there is $\nu\in\mathcal{M}_1$ whose support is included in $(-\infty,0]$ and such that 
$\mu*\nu\in\mathscr{C}$, then $\mu\in\mathscr{C}$.
\end{theorem}
\begin{proof} Let $\mu,\mu_1\in\mathcal{M}_1$ such that for each $n\ge1$, the measures 
$\mu^{*n}$ and $\mu_1^{*n}$ restricted to $[0,\infty)$ coincide. Set $\bar{\mu}=\mu*\nu$ and $\bar{\mu}_1=\mu_1*\nu$. Then 
from commutativity of the convolution product, $\bar{\mu}^{*n}=\mu^{*n}*\nu^{*n}$ and $\bar{\mu}^{*n}_1=\mu^{*n}_1*\nu^{*n}$.
Since the support of $\nu^{*n}$ is included in $(-\infty,0]$ and $\mu^{*n}$ and $\mu_1^{*n}$ restricted to $[0,\infty)$ are known
and coincide, the measures $\bar{\mu}^{*n}$ and $\bar{\mu}^{*n}_1$ restricted to $[0,\infty)$ and known
coincide. Since $\bar{\mu}\in\mathscr{C}$, the measures $\bar{\mu}$ and $\bar{\mu}_1$ are equal. Finally, from Lemma \ref{3024}, 
the characteristic function of $\nu$ does not vanish on any interval of $\mathbb{R}$ and the identity $\mu=\mu_1$ follows
from continuity and injectivity of the Fourier transform. 
\end{proof}
\noindent Theorem \ref{8455} entails in particular that  Conjecture C' is equivalent to the following one: {\it Any distribution  
$\mu\in\mathcal{M}_1$  whose support is not included in $(-\infty,0)$ is determined by its convolution powers $\mu^{*n}$, $n\ge1$ 
restricted to $[a,\infty)$, for some $a\ge0$.} Indeed, it suffices to choose $\nu=\delta_{-a}$ in Theorem \ref{8455}. Finding more 
general examples illustrating this result is an open problem. In order to do so, one needs for instance to find the characteristic function 
$\varphi$ of a random variable which belongs to class $\mathscr{C}$ and the characteristic function $\varphi_Y$ of a nonnegative 
random variable $Y$ such that the ratio $\varphi(t)/\varphi_Y(-t)$ is the characteristic function of some random variable $X$. Then 
since the law of $X-Y$ belongs to class $\mathscr{C}$, so does the law of $X$ from Theorem \ref{8455}.\\

Note that neither Proposition \ref{3555} nor Theorem \ref{8455} allows us to enlarge class $\mathscr{E}$. In order to do so in the 
same spirit as in Theorem \ref{8455}, one needs to find an invertible transformation $T(\mu)\in\mathcal{M}_1$ of a distribution 
$\mu\in\mathcal{M}_1\setminus\mathscr{E}$, such that $T(\mu)^{*n}$, $n\ge1$ restricted to $[0,\infty)$ would be known and 
such that $T(\mu)$ belongs to class $\mathscr{E}$.\\

Let us end this paper with an example of a distribution which satisfies conjecture C, although it does not belong to any of the 
classes studied here. Assume that the support of $\mu$ is included in $\mathbb{Z}$ and recall that according to Lemma \ref{4261},
this assumption can be checked from the data of the measures $\mu^{*n}$, $n\ge0$ restricted to $[0,\infty)$. Assume moreover that 
there are positive integers $a$ and $b$ such that
\begin{eqnarray}\label{8164}
\qquad\left\{\begin{array}{l}
\mbox{$\mu(n)>0$, for all $n\ge a+b$ and $\mu(n)=0$, for all $n=0,\dots,a+b-1$,}\\
\mbox{$\mu^{*2}(n)=0$, for all $n=0,\dots,a$.}
\end{array}\right.
\end{eqnarray}
Then we  can determine $\mu$ on $\mathbb{Z}_-$, so that $\mu\in\mathscr{C}$. Let us first show that $\mu(n)=0$, for all $n\le-b$. 
Assume that there is $n\le-b$ such that $\mu(n)>0$. Then let $k=0,\dots,a$ such that $k-n\ge a+b$. By definition of the 
convolution product $0\le \mu(k-n)\mu(n) \le \mu^{*2}(k)$, but from our assumptions $\mu(k-n)\mu(n)>0$ and  $\mu^{*2}(k)=0$, which is 
contradictory, hence $\mu(n)=0$, for all $n\le-b$.  On the other hand, assumptions  (\ref{8164}) entail that for all $k=a+1,\dots,a+b-1$, 
\[\mu^{*2}(k)=\sum_{i=a+b}^{k+b-1}\mu(k-i)\mu(i)\,,\] 
that is $\mu^{*2}(a+1)=\mu(-b+1)\mu(a+b)$, $\mu^{*2}(a+2)=\mu(-b+2)\mu(a+b)+\mu(-b+1)\mu(a+b+1)$,...
Therefore, this system allows us to determine $\mu(n)$, for $n=-b+1,-b+2,\dots,-1$ and the conclusion follows.\\ 

Let us consider for instance $a=1$, $b=3$ and 
\[\mu(-2)=\mu(-1)=\frac{1-c}2\;\;\;\mbox{and}\;\;\;\mu(n)=\frac1{n^3}\,,\;\;\;n\ge 4\,,\]
where $c=\sum_{n\ge 4}n^{-3}$. Clearly, such a distribution does not belong to any of the classes $\mathscr{A}$, $\mathscr{M}$ or  
$\mathscr{M}_d$. Then let us check that it does not belong to class $\mathscr{E}$. The mean of $\mu$ satisfies
\[\sum_{k\ge-2}k\mu(k)=-\frac32(1-c)+\sum_{n\ge 4}\frac1{n^2}<0\,,\]
so that (\ref{0236}) does not hold. Moreover $\mu$ has no positive exponential moments. Therefore conditions of Theorems \ref{4572}  
and \ref{2258} are not satisfied and since $\mu$ is not downward skip free, we obtain the conclusion. Finally it cannot be proved that $\mu$ 
belong class $\mathscr{C}$ by applying Proposition \ref{3555} or Theorem \ref{8455}. However, from the above arguments, $\mu$ does 
belong to class $\mathscr{C}$.

\vspace*{.5in}

\noindent {\bf Acknowledgement} We would like to thank Rodolphe Garbit, Philippe Jaming and Jean-Jacques Loeb for 
motivating discussions on this subject.

\end{document}